\documentclass[reqno]{amsart}
\allowdisplaybreaks
\usepackage{mathpazo}
\usepackage{amsmath,amsfonts,amssymb}
\usepackage{amsrefs}
\usepackage{amsthm}
\usepackage{latexsym,amsmath,amssymb,amsfonts}
\usepackage{rotating}
\usepackage{booktabs}
\usepackage{xypic} \xyoption{all}
\usepackage{amscd}
\usepackage{hyperref} 
\usepackage{euscript}
\usepackage{hhline}
\usepackage{cite}
\usepackage{color}
\usepackage{graphicx,epstopdf}
\usepackage{epsfig}
\usepackage{xcolor}
\usepackage[all,color]{xy}
\newlength{\fighskip} \fighskip=2pt
\newlength{\figvskip} \figvskip=3pt
\usepackage{hyperref}
\numberwithin{equation}{section}
\DeclareMathOperator{\Diff}{Diff}

\theoremstyle{plain}
\newtheorem{theorem}{Theorem}[section]
\newtheorem{definition-theorem}{Theorem/Definition}[section]
\newtheorem{lemma}[theorem]{Lemma}
\newtheorem{lemma-definition}[theorem]{Lemma/Definition}

\newtheorem{conjecture}[theorem]{Conjecture}

\theoremstyle{definition}
\newtheorem{definition}[theorem]{Definition}

\newtheorem{example}[theorem]{Example}
\newtheorem{prop}[theorem]{Proposition}
\theoremstyle{remark}

\allowdisplaybreaks[4]

\begin{document}
\title{A $C^{\infty}$ closing lemma on torus}
\author{Huadi Qu, Zhihong Xia}
\address{Department of Mathematics, Southern University of Science and
  Technology, Shenzhen, China}
\address{Department of Mathematics, Northwestern University, Evanston,
  IL 60208 USA}
\email{11849458@mail.sustech.edu.cn, xia@math.northwestern.edu}
\date{June 15, 2021, version 1.0}

\maketitle

\begin{abstract} Asaoka \& Irie \cite{Irie2016} recently proved a
  $C^{\infty}$ closing lemma of Hamiltonian diffeomorphisms of closed
  surfaces. We reformulated their techniques into a more general
  perturbation lemma for area-preserving diffeomorphism and proved a
  $C^{\infty}$ closing lemma for area-preserving diffeomorphisms on a
  torus $\mathbb{T}^2$ that is isotopic to identity. i.e., we show
  that the set of periodic orbits is dense for a generic
  diffeomorphism isotopic to identity area-preserving diffeomorphism
  on $\mathbb{T}^2$. The main tool is the flux vector of
  area-preserving diffeomorphisms which is, different from Hamiltonian
  cases, non-zero in general.
  \end{abstract}

%\tableofcontents

\section{Introduction}

 The $C^{r}$ closing lemma is one of the fundamental problems in
 dynamical systems. It goes back to Poincar\'e in his study of the
 restricted three body problem. The problem asks whether the set of periodic
 points of a typical symplectic or volume preserving diffeomorphism
 is dense on a compact manifold. Smale \cite{Smale1998} listed the problem as
 one of the mathematical problems for the 21st century at the end of
 the last century.
 
 Let $M$ be a compact manifold, equipped with a symplectic form
 $\omega$, a closed non-degenerate differential 2-form on $M$. If $M$
 is an orientable surface, then we can take $\omega$ to be an area form. Let
 $\Diff^{r}(M,\omega)$ be the set of $C^{r}$ diffeomorphisms on $M$
 that preserves $\omega$, where $r=1, 2, \ldots, \infty$. For
 $f\in\Diff^{r}(M,\omega)$, denote $\mathcal{P}(f)$ as the set of
 periodic points of $f$. A set in a topological space is said to be
 {\it residual}\/ if it is the intersection of countably many open and dense
 subsets. A property for $C^{r}$ diffeomorphisms of M is said to be
 {\it generic}\/ if there is a residual subset $R$ in $\Diff^{r}(M)$ such that
 the property holds for all $f\in R$.  The $C^r$ closing lemma
 conjectures that for generic $C^r$ area-preserving diffeomorphisms on
 compact manifolds, the set of all periodic points is dense.
 
\begin{conjecture}[$C^r$ closing lemma]
  There exists a residual subset $R\subset\Diff^{r}(M,\omega)$ such
  that for all $g\in R$, The set $\mathcal{P}(g)$ is dense in $M$.
\end{conjecture}

A seemingly weaker, but equivalent statment is the following 
\begin{conjecture}
  For any open subsets $U\subset M$ and $V\subset\Diff^{r}(M,\omega)$,
  there exists $g\in V$ and $x\in U$ such that $x\in\mathcal{P}(g)$.
\end{conjecture}

To see the equivalence, notice that by the Baire Category Theorem, in
a complete metric space, such is $\Diff^{r}(M,\omega)$, any
intersection of countably many open dense sets is again dense, hence the
residual set $R$ is dense in $\Diff^{r}(M,\omega)$. Therefore any open
$V\subset\Diff^{r}(M,\omega)$ contains an element in
$R$, therefore the second conjecture follows from the first one.

On the other hand, if the second conjecture is true, to find the
residual set $R$ in the first conjecture,
let $\{U_i \; | \; i=1,2, \ldots\}$ be a countable basis of
$M$. For each $i \in \mathbb{N}$, define
$$H_i=\{\psi\in\Diff^{r}(M,\omega)\; | \; \psi  \mbox{ has a
  non-degenerate periodic point in }U_i\}$$
Then $H_i$ is open. By the second conjecture, $H_i$ is also dense in
$\Diff^{r}(M,\omega)$. So $R=\bigcap_{i} H_i$ is a residual subset in
$\Diff^{r}(M,\omega)$ with the desired property. \hfill \qed

\vspace{1em}
The $C^r$ closing lemma conjecture was in essence made by Poincar\'e
\cite{Poincare19931}, and it remains open up to now in general cases. However,  for
$r=1$, i.e. the $C^1$ closing lemma, it has been completely settled, for
symplectic and volume-preserving
difformorphisms of compact manifolds (cf.\ Pugh \cite{Pugh1967}, Liao
\cite{Liao1979}, Pugh and Robinson \cite{Pugh1983}). The $C^r$ closing lemma for $r>1$
has eluded mathematicians for generations. However, important progress
has been made. For example, the Anosov closing
lemma for uniformly hyperbolic difformorphisms (Anosov
\cite{Anosov1962, AnosovZhuzhoma2012}) and
non-uniformly hyperbolic diffeomorphisms (Pesin \cite{Pesin1977}), a class of
partially hyperbolic diffeomorphisms (Xia and Zhang\cite{XiaZhang2006}) and other
cases. For more details, we refer to Saghin and Xia \cite{RaduXia2006}.
 
Following recent exciting development in contact geometry (cf.\
Hutchings\cite{Hutchings2014}\cite{Hutchings2016}), an important and
remarkable progress has been made for $C^r$ closing lemma for
Hamiltonian diffeomorphisms on surfaces. A diffeomorphism is said to
be {\em Hamiltonian}\/ if it is a time-1 map of a 1-periodic
time-dependent Hamiltonian system. Asaoka and Irie \cite{Irie2016}
showed that a generic $C^r$, $r=1, 2, \ldots, \infty$, Hamiltonian
difformorphism of a closed surface, has a dense set of periodic
orbits. This result comes as an interesting application of spectral
invariants in embedded contact homology, known as $ECH$ invariants.

Our goal is to extend these results to a more general area preserving
diffeomorphism that is not coming from periodic Hamiltonian flow. As a
first step in this direction, we restrict ourselves to torus
$\mathbb{T}^2$ and to area preserving diffeophisms that are isotopic to identity.
We show that the $C^\infty$ closing lemma holds in this case. More
precisely, let $\text{Symp}_{0}^{\infty}(\mathbb{T}^2,\omega)$ be the
set of all $C^\infty$ area-preserving diffeomorpism on $\mathbb{T}^2$ that
is isotopic to identity, we have

\begin{theorem}
  There exists a residual subset
  $R\subset\text{Symp}_{0}^{\infty}(\mathbb{T}^2,\omega)$ such that
  for all $g\in R$, the set $\mathcal{P}(g)$ is dense on $\mathbb{T}^2$.
\end{theorem}

We remark that the set of Hamiltonian diffeomorphisms is a proper
subset of $\text{Symp}_{0}^{\infty}(\mathbb{T}^2,\omega)$,
characterized by zero flux, or preservation of center of gravity in
Arnold's term. There are large classes of area-preserving maps that
are not Hamiltonian.  We will provide more details on this later, but
first,

\begin{example} Using the canonical chart
  $\mathbb{T}^2=\mathbb{R}^2/\mathbb{Z}^2$ and the standard symplectic
  form $\omega=d p\wedge d q$.
Consider on $(\mathbb{T}^2, d p\wedge d q)$ the following simple system:
 \begin{equation*}
 \begin{cases} 
    \dot{p}=& \alpha \\
    \dot{q}=& \beta 
 \end{cases}
\end{equation*}
The flow is area-preserving, but it is Hamiltonian only when $\alpha =
\beta =0$. Similarly, symplectomprphism $f(p,q)=(p+ \alpha,q+\beta )$
is Hamiltonian only if  $(\alpha, \beta) \in\mathbb{Z}^2$.
\end{example}

\

To prove our theorem, we need a more general perturbation technique in
the context of the area-preserving maps. We state this new
perturbation lemma as Theorem \ref{thm2.3}. The advantage of this
theorem is that one can directly verify various conditions
from the area-preserving maps without going through underlying contact
flow. This perturbation lemma can be useful in many other settings.

\section{A review of recent results}
In this section we review Asaoka and Irie's result \cite{Irie2016} for
$C^\infty$ Hamiltonian diffeomorphisms of closed surfaces. The
techniques of their proof is essential for our results.

Consider a compact, oriented surface $M$ and a symplectic form
$\omega$ on it. Let $\Diff^{\infty}(M,\omega)$ denote the group of
$C^{\infty}$ diffeomprphisms on $M$ that preserves $\omega$, equipped
with the $C^{\infty}$ topology. For any time-periodic smooth function
on $M$, $H_t: S^1\times M\to \mathbb{R}$, where
$S^1 = \mathbb{R} / \mathbb{Z}$, the associate Hamiltonian vector
fields is $X_{H_t}$, satisfying
$$\iota_{ X_{H_t}}\omega = \omega(X_{H_{t}}, \cdot ) = -dH_t.$$
We denote with $$\phi ^{t}_H:\mathbb{R}\times M\to M,\;\; 
(t,x) \mapsto \phi^{t}_{H}(x)$$
the flow generated by the vector field $X_{H_t}$.
Notice $\phi ^{t}_H$  is a one-parameter subgroup in the group $\Diff^{\infty}(M,\omega)$. 
$\phi^{1}_H$ is the Poincar\'e map on the cross section $\{0\} \times
M$. The fixed points of $\phi^{1}_H$ correspond exactly to 1-periodic
solutions of the vector field $X_{H_t}$.

A symplectic diffeomorphism $f \in \Diff(M,\omega)$ is said to be {\em
  Hamiltonian} if there is time-periodic Hamiltonian function  $H_t:
S^1\times M\to \mathbb{R}$ such that $f = \phi^{1}_H$. We denote the
set of all Hamiltonian diffeomorphisms by
$\text{Ham}(M,\omega)$. Obviously 
$$\text{Ham}(M,\omega) \subset \Diff(M,\omega), $$
the inclusion is proper for most manifold $M$.

For simplicity, we will restrict ourselves to $C^\infty$
diffeomorphisms. Results for $C^r$ topology, $r\in \mathbb{N}$,
follows. From now on in our notation, both $\Diff(M,\omega)$ and
$\text{Ham}(M,\omega)$ are equipped with $C^{\infty}$ topology on
$\Diff(M)$.

Hamiltonian diffeomorphisms enjoy some fine structures critical
to variational techniques in symplectic geometry. These fine
structures give rise to many inetresting dynamical properties. An
important example is the Arnold conjecture for Hamiltonian
diffeomorphisms. Another good example is the following
Asaoka \& Irie's $C^\infty$ closing lemma for Hamiltonian
diffeomorphisms on compact surfaces \cite{Irie2016}.

\begin{theorem}[Asaoka \& Irie]
  Let $M$ be any closed, oriented surface with an area form $\omega$,
  $g\in\text{Ham}(M,\omega)$. Let $\mathcal{P}(g)$ be the set of all
  periodic points of $g$. For any nonempty open set $U\subset M$,
  there exists a sequence $(g_j)_{j\geqslant1}$ in
  $\text{Ham}(M,\omega)$ such that
  $\mathcal{P}(g_j)\cap U\neq\emptyset$ and
  $\lim_{j \to \infty }{g_j}=g$.
\label{thm2.1}\end{theorem}

 The key idea in the proof is to embed the area-preserving
 diffeomorphism on the surface into a three dimensional contact flow,
 where one can apply Irie's perturbation lemma for contact flow
 \cite{Irie2015}.

 Let $(M,\lambda)$ be a contact manifold, where $\lambda$ denotes the
 contact form with the contact distribution
 $\xi_{\lambda}:=ker\lambda.$ The Reeb vector field $R_{\lambda}$ is
 defined by $\lambda(R_{\lambda})=1$, $d\lambda(R_{\lambda},\cdot)
 \equiv 0$. Let
 $\mathcal{P}(M,\lambda)$ denote the set of periodic orbits of
 $R_\lambda$:
 $$\mathcal{P}(M,\lambda): =
 \{\gamma:\mathbb{R}/T_{\gamma} \mathbb{Z}\rightarrow
 M \; |\; T_{\gamma}>0, \; \dot{\gamma} = R_{\lambda}(\gamma)\}.$$
 here $T_{\gamma}$ is the period of orbit $\gamma$ of the vector field
 $R_{\lambda}$.
 
 \begin{lemma}[Irie's perturbation lemma]
Let $(M,\lambda)$ be a closed contact 3-manifold, For any $h\in
C^{\infty}(M,\mathbb{R}_{\geqslant0})\backslash\{0\}$, there exist
$t\in[0,1]$ and $\gamma\in\mathcal{P}(M,(1+th)\lambda)$ which
intersects supp($h$).
\end{lemma}

Here $\mbox{supp}(h)\subset M$ is the support of $h$.  The proof of
this lemma is based on embedded contact homology (ECH), which is
defined in terms of periodic orbits of the Reeb flow. The key property
here for ECH and its spectral invariants is that they recover the
contact volume (Gardiner, Hutchings and Ramos \cite{Hutchings2015}). This
remarkable result implies that one can perturb the structure of
periodic orbits by simply perturb contact volume. It is to see that
adding $h\lambda$ with $h >0$ to contact form $\lambda$, one can
increase the contact volume.

Irie's generic density theorem for periodic Reeb orbits for contact
flow follows easily from the above perturbation lemma \cite{Irie2015}.

\section{A new perturbation lemma}
\label{section3}

Next, we would like to introduce a perturbation lemma in the context
surface diffeomorphisms, without going through contact flows. We will
use quantities and conditions directly verifiable, and computable, in
area-preserving maps. We will use this perturbation theorem to prove
our main theorem.

Let $U$ be a simply connected region in $M$, homeomorphic to a closed
disk in $\mathbb{R}^2$, with the proper radius such that $U$ has the
same area as the disk. We may assume, for simplicity and without loss
of generality, that $U$ is such a disk. Let $h$ be a $C^\infty$ area-preserving
diffeomorphism of $U$ such that $h$ is identity for all points near
the boundary of $U$.  Since $h$ is area-preserving, the differential
1-form $h^*(ydx) - ydx$ is exact, therefore, there is a function $g_h$
on the disk $U$ such that
$$dg_h = h^*(ydx) - ydx$$
The function $g_h$ is unique up to an integral constant, we fix the
constant such that $g_h=0$ on the boundary of $U$.

The function $g_h$ is called the action for the diffeomorphism $h$.
We define the total action by integrating $g_h$ over $U$:
$$A(h) = \int_U g_h dxdy.$$

Let $\lambda$ be another 1-form such that $d\lambda = \omega$, then
$\lambda -ydx = dS$ for some real valued function $S$ on $U$. Let
$g_h', A'(h)$ be the action function and total action, respectively,
on $U$ corresponding to $\lambda$, with the same property that
$g_h' =0$ on $\partial D$. Then,
$$g_h'(x) - g_h(x) = S(f(x)) - S(x)$$
and $$A'(h) - A(h) = \int_U S(h(x)) \omega- \int_U  h(x) \omega =0.$$
The last equality is due to the fact that $h$ preserves $\omega$. We
conclude that the total action $A(h)$ is independent the choice of
1-form $ydx$.

Let $e: U \rightarrow M$ be an area-preserving embedding of $U$ into
an closed surface $M$. We can extend $h$, by abusing the notation, to
an area-preserving diffeomorphism of $M$, such that $h = \mbox{id}$
outside of $U$. Clearly, $h$ is Hamiltonian and the action $A(h)$ is
independent of the embedding.

We can now state our perturbation lemma. 

\begin{theorem}[A perturbation lemma]\label{thm2.3}
  Let $M$ be a closed, oriented surface with an area-form and let $f$
  be Hamiltonian diffeomorphism on $M$.
  \begin{itemize}
  \item 
  Let $U_1$, $U_2$, $\ldots$,
  $U_k$ be finite disjoint closed disks on $M$.

\item For all $t\in [0, 1]$, Let $h^t_1, h^t_2, \ldots, h^t_k$ be
  families of area-preserving diffeomorphisms supported on $U_1$,
  $U_2$, $\ldots$, $U_k$, respectively. Assume that
  $h^0_1, h^0_2, \ldots, h^0_k$ are all identity maps on respective
  domains.

\item For each $t \in [0, 1]$, let
  $A(h^t_1), A(h^t_2), \ldots, A(h^t_k)$ be their respective total
  actions.

\item Let $U = U_1 \cup U_2 \cup \ldots \cup U_k$, for any
  $t \in [0, 1]$, let
  $$h^t = h^t_1 \circ h^t_2 \circ \ldots \circ h^t_k: U \rightarrow
  U$$ be a family of area-preserving diffeomorphisms.
    
\end{itemize}

Suppose that at $t=1$, $$A(h^1_1) + A(h^1_2) + \ldots +  A(h^1_k) \neq
0$$

Then there exists $t \in [0, 1]$, such that $$h^t \circ f: M
\rightarrow M$$
has a periodic point in $U$.
\label{thm2.3}\end{theorem}

The proof of this theorem follows the same idea as in
Asaoka-Irie theorem \cite{Irie2016}. One needs to note that the total volume in
contact manifold is exactly the total action for area-preserving
map. As we mentioned earlier, ECH invariants,
defined by periodic orbits of the map, recover the volume of the
contact manifold. Any change in contact volume, or the total action
for surface diffeomorphisms, forces changes in the set of periodic
points.

We need to be more precise on action function. An area-preserving
diffeomorphism on a surface compact $M$ with area form $\omega$ is
said to be {\em exact}\/ if there is a 1-form $\lambda$ with
$d\lambda = \omega$ such that $f^*\lambda - \lambda = dg$ for some
real valued function $g$ on $M$. This function $g$ is called the {\em
  action}. Once $\lambda$ is fixed, it is defined upto an integration
constant. This constant can be fixed by assigning a spefic value at
any given point. The total action is defined to be
$$A(f) = \int_M g \omega.$$

The total action depends on the choice of 1-form $\lambda$, through
the first cohomology of the manifold $M$.

Fix a 1-form $\lambda$, a useful fact is that the action is additive
for compositions of exact symplectic diffeomorphisms. 

\begin{prop} \label{prop2.4}  Let $f_1$ and $ f_2$ be exact
  symplectic diffeomorphisms of $M$. Let 
$g_1$, $g_2$ and $g_{12}$ be action functions for $f_1$, $f_2$ and $f_2
\circ f_1$ respectively. Suppose there is a point $x_0 \in M$ such that
 $$g_{12}(x_0) = g_1(x_0) + g_2(f_1(x_0)).$$

Then $$A(f_2 \circ f_1) = A(f_1) + A(f_2).$$ 
\end{prop}

This lemma shows that change in action with a perturbation can be
easily calculated.

\begin{proof} The action function $g_{12}$ is define by
$$dg_{12} = (f_2  \circ f_1)^* \lambda -\lambda,$$
hence,
$$dg_{12} =\left\{ f_1^*(f_2 ^* \lambda) - f_2 ^* \lambda \right\} + \{f_2 ^* \lambda- \lambda\}$$
$$=d\tilde{g}_1 + d g_2$$
for some function $\tilde{g}_1$. In fact  $$d\tilde{g}_1 = f_1^*(f_2 ^* \lambda) - f_2 ^* \lambda =
\{f_1^*(f_2^* \lambda- \lambda) - (f_2 ^* \lambda - \lambda)\} + \{f_1^*\lambda -
\lambda \}$$
$$=\{ f_1^*dg_2- dg_2\}  + dg_1= d(g_2\circ f_1-g_2) + dg_1.$$
Obviously, such $\tilde{g}_1$ exists, we may choose
$$\tilde{g}_1 = (g_2\circ f_1 -g_2) + g_1,$$
and then we may choose  $$g_{12} = \tilde{g}_1 + g_2 = (g_2\circ f_1 -g_2) + g_1 +
g_2 =  g_2\circ f_1  + g_1,$$
i.e., we may choose the action function for $f_2\circ f_1$ by adding
the action function for $f_1$ and the $f_1$-shifted action function for $f_2$. 
Under this choice, we have \begin{eqnarray} {g}_{12}(x_0) =  g_1(x_0) +
g_2(f_1(x_0)) \label{eq1}. \end{eqnarray}

Now \begin{eqnarray} A(f_2 \circ f_1) &=& \int_M g_{12} \omega \nonumber \\
&=& \int_M (g_1 + g_2\circ f_1)
\omega \nonumber \\
&=& \int_M g_1 \omega + \int_M (g_2\circ f_1) 
\omega \nonumber \\
&=& \int_M g_1 \omega + \int_M g_2 (f_1^* 
\omega) \nonumber \\
&=& \int_M g_1 \omega + \int_M g_2
\omega \nonumber \\
&=& A(f_1) + A(f_2). \nonumber \end{eqnarray}
We remark that the above equality is independent of the choice of
1-form $\lambda$ and therefore the choices of
the action functions, as long as condition (\ref{eq1}) holds.

This proves the proposition.
\end{proof}

There is one issue concerning the action function: it is not well
defined for symplectic diffeomorphisms on compact manifold without
boundary.  Hamiltonian diffeomorphisms are exact provided that
$\omega$ has a primitive 1-form $\lambda$. It is well-known that no
such 1-form exists for compact symplectic manifold without boundary.

To overcome this difficulty, Asaoka-Irie's idea is to blow-up a fixed
point on compact surface. The existence of such fixed point is
guaranteed by Floer's proof of Arnold conjecture.  Moreover, in order
to fit into the contact manifold framework, one needs to construct the
so-called {\em open book decomposition}\/.

Theorem \ref{thm2.3} follows from Asaoka-Irie's open book
decomposition, Proposition \ref{prop2.4}, and, of course, the volume
theorem of Gardiner, Hutchings and Ramos \cite{Hutchings2015}.

We conclude this section by showing how a local perturbation can
change action.

\begin{example}
 Let $D^2$ be a closed unit disk in $\mathbb{R}^2$, we construct a
 family of $C^\infty$ area-preserving map $h_t$ on $D^2$, such that
 \begin{itemize}
 \item $h_{0}=id$;
   \item The action function $g_t \equiv 0$ on $\partial D^2$ for all $t \in
     [0, 1]$;
     \item $A(h_{1})\neq 0$.    
     \end{itemize}
     
 Let $b: [0,1] \rightarrow \mathbb{R}$ be a $C^\infty$ bump
 function, such that
$$b(r) = \left\{ \begin{matrix} 1, & 0 \leq r \leq  \frac{1}{3}   \cr
     0,  & \frac{2}{3} \leq r \leq  1 \end{matrix} \right.$$
and $b(r) \geq 0$, $b'(r) \leq 0$ for  all $r\in[0,1]$.

 Take the standard area form $\omega$. Under the polar coordinate
 $(r,\theta)$ on $D^2$,
 $\omega=rdrd\theta$, we take $\lambda =\frac{r^2}{2}d\theta$. Define a family of
 diffeomorphisms $h_{t}(r,\theta)=(r,\theta+tb(r))$. Then we have
$$h_{t}^{*}\beta-\beta=\frac{r^2}{2}d(\theta+tb(r))-\frac{r^2}{2}d\theta
                                =\frac{r^2}{2}tb'(r)dr=dg_{h_{t}}.$$
                               
Hence the action function is a path independent line integral
                               $$g_t=g_t(r, \theta)=\int_{(1, 0)}^{(r,
  \theta)}\frac{r^2}{2}tb'(r)dr +g(1, 0) = t\int_1^{r}\frac{r^2}{2}b'(r)dr.$$
 Here we set $g(1, 0) = 0$. Finally we have $g(r, \theta) \geq 0$ for all $(r,
 \theta)$ and the total action
$$A(h_t) = t\int_{r=0}^1 \int_{\theta =0}^{2\pi} \big( \int_0^r
\frac{s^2}{2} b'(s) ds\big) r d\theta dr \geq 0.$$
In particular, $A(h_0)=0$ and $A(h_t) > 0$ for any $t > 0$.
                             
\end{example}

\section{Flux}

Our goal is to generalize the $C^\infty$ closing lemma to certain
non-Hamiltonian area-preserving diffeomorphisms. The key property that
distinguishes a Hamiltonian diffeomorphism from a general symplectic
deffeomorphism involves the concept of flux. We first give a general
definition of flux, even though the flux in our specific application
is much simpler.

By a theorem of Weinstein in \cite{Weinstein1971}, the Lie group $Symp(M,\omega)$,
for compact symplectic manifold M, is locally path-connected.  Let
$Symp_{0}(M,\omega)$ denote the path component of the Lie group of
symplectomorphisms of $(M,\omega)$ containing the identity, for every
$\psi\in Symp_{0}(M,\omega)$ there exists a smooth family of
symplectomorphisms $\psi_t\in Symp(M,\omega)$ such that $\psi_{0}=id$
and $\psi_{1}=\psi.$ For such a family of symplectomorphisms there
exists a unique family of vector fields $X_t:M\rightarrow TM$ such
that
$$\frac{d}{dt}\psi_{t}=X_{t}\circ\psi_{t}.$$  
Since $\psi_t$ preserving the symplectic structure we have $L_{X_t}\omega=0$ hence 
$\iota(X_t)\omega=\lambda_t$ is a family of closed 1-forms.

First consider a closed loop $\{\phi_t\}$ in $Symp_{0}(M,\omega)$ with
$\phi_{0}=\phi_{1}=id$, let $\{\lambda_{t}\}$ be the family of closed
1-forms generationg this loop.
\begin{definition}
  The flux of this loop is given by
$$flux(\{\phi_t\})=\int_{0}^{1}[\lambda_t]dt\in H^{1}(M,\mathbb{R}).$$
\end{definition}

Under the usual identification of $H^{1}(M;\mathbb{R})$ with
$Hom(\pi_{1}(M);\mathbb{R})$, the above cohomology class corresponds
to the homomorphism $\pi_{1}(M)\rightarrow\mathbb{R}$ defined by
$$l\mapsto\int_{0}^{1}\int_{0}^{1}\omega(X_{t}(\l(s)),\dot{l}(s))dsdt $$
for $l: S^{1}=\mathbb{R}/\mathbb{Z}\rightarrow M\in\pi_{1}(M).$ Use this fact we
can prove that the right hand side only depends on the homotopy class
of the path $\phi_{t}$ in $\pi_{1}(Symp_{0}(M,\omega))$.  Thus we get
a homomorphism
$$flux:\pi_{1}(Symp_{0}(M,\omega))\rightarrow H^{1}(M,\mathbb{R})$$
The image of this homomorphism $\Gamma\subset H^{1}(M,\mathbb{R})$ is
called the \textbf{Calabi group}, or \textbf{flux group}.

Next, given any $\psi\in Symp_{0}(M,\omega)$, any path  $\{\psi_t\}\in Symp_{0}(M,\omega)$
that connected $\psi$ and identity also determined a family of vector fields and a family of closed
1-forms. Then we can define its flux in the same way
$$flux(\{\psi_t\})=\int_{0}^{1}[\lambda_t]dt\in H^{1}(M,\mathbb{R}).$$
notice that the result will depend on the choice of homotopy type of the path connecting
$id$ with $\psi$, but the difference between two choices belongs to $\Gamma$. 
Hence we have
$$flux: Symp_{0}(M,\omega)\rightarrow H^{1}(M,\mathbb{R})/\Gamma.$$

We give a geometric description of the flux homomorphism. The value of
$flux(\{\psi_t\})$ on a 1-cycle $l$ on M is the symplectic area swept
out by the path of $l$ under the isotopy $\{\psi_t\}$. More precisely,
let $\partial[l]=\cup_{t}\psi_{t}(l)$ be a 2-cycle which is the image
of $l$ under $\psi_t$, then
$(flux(\{\psi_t\}),[l])=([\omega],\partial[l])$ for all
$l\in H_{1}(M,\mathbb{Z}).$

Hence the flux group is a subgroup of $H^{1}(M;P_{\omega})$, where
$P_{\omega}:=([\omega], H_{2}(M;\mathbb{Z}))$ is the
countable group of periods of $\omega$, i.e., the set of values taken
by $\omega$ on the countable group $H_{2}(M;\mathbb{Z})$ of integral 2-cycles.

The flux homomorphism for a symplectic manifold $(M,\omega)$ was first
introduced by Calabi in \cite{Calabi1970} in 1970, and studied futher by
Banyaga in \cite{Banyaga1978}. Banyage showed that $ker(F)$ is exactly
$Ham(M,\omega)$, thus we have an exact sequence 
$$0\rightarrow Ham(M,\omega)\rightarrow Symp_{0}(M,\omega) \rightarrow
H^{1}(M,\omega)/\Gamma\rightarrow0.$$ 

That is, we have the following lemma:
\begin{lemma}
  For $\psi\in Symp_{0}(M,\omega)$, it is Hamiltonian if and only if
  there exists a symplectic isotopy
$$[0,1]\rightarrow Symp_{0}(M,\omega): t\mapsto\psi_t $$
such that
$\psi_{0}=id, \psi_{1}=\psi, and flux(\{\psi_{t}\})=0\text{ in
}H^{1}(M,\mathbb{R}).$

Moreover, if $flux(\{\psi_{t}\})=0$ then $\psi_t$ is isotopic with fixed endpoints to a Hamiltonian isotopy.
\end{lemma}

This lemma states that zero flux is neccesary, but not sufficient, for a
symplectic diffeomorphism to be Hamiltonian. The sufficient condition
is that there is a zero flux isotopy.  For a proof of this lemma, see
Banyaga \cite{Banyaga1978}.

We now consider two dimensional case, where the group of
symplectomorphisms are those area-preaserving diffomorphisms that
preserve the orientation. We normalize the area form $\omega$ by
assume that $\int_{M}\omega=1$.  In this case the flux group is
exactly $\mathbb{Z}$, since $P_{\omega}$, the set of values taken by
$\omega$ on the countable group $H_{2}(M;\mathbb{Z})$ of integral
2-cycles is exactly $\mathbb{Z}$. Hence we can simply define the flux
of a diffeomorphism $f\in\Diff_{0}(M,\omega)$ in the following way.

Let $l$ be an oriented closed curve in $M$, then there is an oriented
disk $D\subset M$ such that the boundary $\partial D=f(l)-l$. The flux
of $f$ across $l$ can be calculate as
$$F_{f}(l)=\int_{D}\omega \ \ mod\ 1.$$
Note that the choice of $D$ is up to the second homology class of $M$,
the above quantity is defined up to mod 1, therefore independent of
the choice of $D$. When $f$ is isotopic to identity, $D$ can be
uniquely defined by the isotopy.

As we see before, the flux depends only on the homology class of $l$,
hence $F_{f}: H_{1}(M,\mathbb{R})\to\mathbb{R}\ \ \ mod\ 1 $ can be
represented by a cohomology vector.

Consider an orientable compact surface $M$ of genus g with $g\geqslant
1$. Let $(a_i,b_i), i=1,2,...g$ be the canonical generators of its
first homology $H_{1}(M_g,\mathbb{R})$,  let
$V_{f}=(F_{f}(a_1),F_{f}(b_1),F_{f}(a_2),F_{f}(b_2),...F_{f}(a_g),F_{f}(b_g))$
be the associate vector in $H^{1}(M,\mathbb{R})$. We call $V_f$ the
{\em flux vector}\/ of $f$.  The flux vector has a nice additive property under the iterations of $f$, i.e.,  $F_{f^k}=kF_{f}$.  
%We remark that this definition of flux vector could be extended to the some cases where the map is not isotopic to identity, e.x, for a closed curve that is homologous to its image.

The key fact about flux vector is the following lemma:

\begin{lemma}
Let $f\in\Diff_{0}(M,\omega)$ be an area preserving diffeomorphism
that isotopic to identity, then for any neighborhood $V$ of $f$ in
there is a map $g$ in $V$ such that the flux vector $v_g$ is
rational.
\end{lemma}

\begin{proof}
Let $(a_i,b_i)$ be the canonical generators of the first homology
$H_{1}(M,\mathbb{Z})$. For any $i=1,2,...g$, we may assume that
$a_i$ and $b_i$ are simple closed curves such that these curves don't
intersect each other except $a_i$ and $b_i$; $a_i$ and $b_i$ intersect
at only one point and the intersection is transversal.

Fix $i$, we choice a small tubular neighborhood of $b_i$, without loss
of generality, we may parametrize this tubular neighborhood by an
area-preserving parametrization: 
$$\delta_{b_i}: S^{1}\times[-\delta,\delta]\rightarrow M$$ 

Let  $\beta:\
[-\delta,\delta]\rightarrow\mathbb{R}$ be a $C^{\infty}$
function such that $\beta(-\delta)=\beta(\delta)=0,\ \beta(t)>0$ for
all $-\delta<t<\delta$, and all the derivatives of $\beta(t)$ at
$\pm\delta$ are zero. Let
$$T_{\varepsilon}: S^{1} \times [-\delta,\delta] \rightarrow M, \;\;
(\theta,t) \mapsto (\theta+\varepsilon\beta(t))$$ be a small
perturbation on the tubular neighborhood. 

Let $h_{\varepsilon}: M \rightarrow M$ be a diffeomorphism
defined as 
$$h_{\varepsilon}(z) =
\begin{cases} z & z\notin \delta_{b_i} \\
\delta_{b_i} \circ T_{\varepsilon}\circ(\delta_{b_i})^{-1}(z) & z\in\delta_{b_i}
\end{cases}$$ 
Then $h_{\varepsilon} \rightarrow Id_M$ in $C^{\infty}$
topology as $\varepsilon \rightarrow 0$, hence
$f\circ h_{\varepsilon}\rightarrow f$ as $\varepsilon\rightarrow0$.

It is easy to see that the flux of $h_{\varepsilon}$ across $a_j$ and
$b_j$ are all zero for $j\neq i$ and the flux across $b_i$ is also
zero. We have infinitely many choices of $\epsilon$ in
$F_{h_{\varepsilon}}(a_i)\neq0$ such that
$F_{h_{\varepsilon}}(a_i)\neq0$, hence  
$F_{f\circ h_{\varepsilon}}(a_i)=F_{f}(a_i)+F_{h_{\varepsilon}} (a_i)$ could be choice to be rational.

We can do similar perturbations to every closed curves $a_i$ and $b_i$
so that the flux across $a_i$ and $b_i$ are all rational.
\end{proof}

\section{proof of the main result}

In this section we consider
$f\in \Diff ^{\infty}_{0}(\mathbb{T}^{2},\omega)$, the area-preserving
difformorphism of $\mathbb{T}^{2}$ that homotopic to the identity.
Identify $\mathbb{T}^{2}$ with the quotient
$\mathbb{R}^{2}/\mathbb{Z}^{2}$, and let
$\pi: \mathbb{R}^{2}\rightarrow\mathbb{T}^{2}$ be the associated
covering map.  For each continuous map
$f:\mathbb{T}^{2}\rightarrow\mathbb{T}^{2}$ we can lift it to
$\mathbb{R}^2$, to a continuous map
$\widetilde{f}:\mathbb{R}^{2}\rightarrow\mathbb{R}^{2}$ such that
$f\circ\pi=\pi\circ\widetilde{f}$.  If $\widetilde{f_1}$,
$\widetilde{f_2}$ are two lifts of $f$ then there exists
$z\in\mathbb{Z}^{2}$ such that
$\widetilde{f_1}(v)=\widetilde{f_2}(v)+z$ for all
$v\in\mathbb{R}^{2}$.

 Any diffeomorphism $f:\mathbb{T}^{2}\rightarrow\mathbb{T}^{2}$ that is isotopic to identity can be written in the form
$$f(x,y)=(x+\alpha(x,y)),y+\beta(x,y))$$
where the functions $\alpha$ and $\beta$ are of period 1 in both
variables $x,y\in\mathbb{R}$. For this type of diffeomorphisms there exist a lift 
$$\widetilde{f}(x,y)=(x+\widetilde{\alpha}(x,y),y+\widetilde{\beta}(x,y)).$$
such that
$\widetilde{f}(x+z_{1},y+z_{2})=\widetilde{f}(x,y)+(z_{1},z_{2})\text{
  for }z=(z_{1},z_{2})\in\mathbb{Z}^2.$

 \begin{lemma}
   Let $f_{t}\in Symp_{0}({\mathbb{T}^2},\omega)$ be a symplectic
   isotopy with a lift
   $\widetilde{f_t}: \mathbb{R}^2\rightarrow\mathbb{R}^2$ such that
 $$\widetilde{f_t}(x+z_1,y+z_2)=\widetilde{f_t}(x,y)+(z_1,z_2)$$ 
 for $(x,y)\in\mathbb{R}^2$ and $(z_1,z_2)\in\mathbb{Z}^2$. Then 
 $$Flux(\{\widetilde{f_t}\}) = a_1dx+a_2dy, (a_1,a_2) =
 J_{0}\int_{\mathbb{T}^2}(\widetilde{f_1}(x, y)-(x,y))dxdy$$
 Here $J_{0}=
 \begin{pmatrix}
   0 & -1 \\
   1 & 0
\end{pmatrix}$ is the Poisson matrix.

Hence the flux homomorphism for the torus $(\mathbb{T}^2,\omega)$
descends to a homomrphism $\rho:
\text{Symp}_{0}(\mathbb{T}^2,\omega)\rightarrow
H^{1}(\mathbb{T}^2;\mathbb{R})/H^{1}(\mathbb{T}^2;\mathbb{Z})$ and
$\Gamma=H^{1}(\mathbb{T}^2;\mathbb{Z})$. 
 \end{lemma}
 
 \begin{proof}
   Let $H_t$ be a smooth family of generating Hamiltonians for $f_t$, so that \\
   $\frac{d}{dt}\widetilde{f_t}=-J_{0}\nabla
   H_{t}\circ\widetilde{f_t}$. We claim that there are functions
   $h_{j}(t),\ j=1,2$ such that
$$H_{t}(x+z_{1},y+z_{2})=H_{t}(x,y)+h_{1}(t)z_{1}+h_{2}(t)z_{2}$$
for $0\leq t\leq1$. To see this, define $h_{j}(t)$ so that the
equation holds when $(z_{1},z_{2})$ is the 1,2 th vector of the
standard basis, and then use linearity with respect to
$(z_{1},z_{2})$. Hence  
\begin{eqnarray*}
   a&=&J_{0}\int_{\mathbb{T}^2}(\widetilde{f_1}(x,y)-(x,y)) dxdy \\
     &=&J_{0}\int^{1}_{0}\int_{\mathbb{T}^2}\frac{d}{dt}\widetilde{f_t}(x,y) dx dy dt \\
     &=&\int^{1}_{0}\int_{\mathbb{T}^2}\nabla H_{t}\circ\widetilde{f_t}(x,y) dx dy dt \\
     &=&\int^{1}_{0}\int_{\mathbb{T}^2}\nabla H_{t}(x,y) dx dy dt \\
     &=&\int^{1}_{0}h(t)dt
\end{eqnarray*}
Furthermore, although the functions $H_{t}$ are not defined on
$\mathbb{T}^2$, both $dH_{t}$ and $X_{H_{t}}$ descend to
$\mathbb{T}^2$. Thus we find  
\begin{eqnarray*}
flux(\{f_{t}\})&=&\int^{1}_{0}[\iota(X_{H_{t}})dxdy]dt \\
                       &=&\int^{1}_{0}[dH_{t}]dt \\
                       &=&\int^{1}_{0}h_{1}(t)dtdx+\int^{1}_{0}h_{2}(t)dtdy \\
                       &=&a_{1}dx+a_{2}dy
\end{eqnarray*}
\end{proof}
  
\begin{definition}
 A symplectomorphism $f$ of the torus is called {\em exact}\/ if it admits a
 lift $\widetilde{f}$ such that  
 $$\widetilde{f}(x+z_{1},y+z_{2})=\widetilde{f}(x,y)+(z_{1},z_{2}),\ \
 \int_{{T}^{2}}(\widetilde{f}(x,y)-(x,y))dxdy=0.$$ 
\end{definition}

We remark that, the condition on the lift of $f$ on
$\mathbb{T}^2$ implies that $f$ is isotopic to identity.

 Now we are ready to prove the main result.
 
 \textbf{Proof of theorem 1.3}

 Let $f$ be a symplectic diffeomorphism on $\mathbb{T}^2$,
 isotopic to identity. Hamiltonian diffeomorphism is
characterized by the fact that there exists a lift $\widetilde{f}$
such that 
$$\int_{\mathbb{T}^2}(\widetilde{f}(x,y)-(x,y))dxdy=0.$$

According to lemma 3.2, we may assume, by making a small perturbation,
that $f$ has a rational flux vector. That is,
 $$V_{f}=(F_{f}(a_1),F_{f}(b_1))=(\frac{p_1}{q},\frac{p_2}{q}) \text{
   for some }p_1,p_2,q\in\mathbb{Z}.$$ 
 
 The flux vector of $f$ is computed in terms of a lift $\widetilde{f}$:
$$V_{\widetilde{f}}=(\int_{\mathbb{T}^2}\widetilde{\alpha}(x,y)dxdy,
\int_{\mathbb{T}^2}\widetilde{\beta}(x,y)dxdy)$$ 
and $V_{f}=V_{\widetilde{f}}\ \text{\ mod 1}.$

Hence we can write
$V_{\widetilde{f}}=(\frac{p_1}{q},\frac{p_2}{q})$. Recall that the
flux vector has an additive property,
$F_{\widetilde{f}^k}=kF_{\widetilde{f}}$ for all $k\in\mathbb{Z}$, so
we have
$V_{\widetilde{f}^q}=qV_{\widetilde{f}}=(p_1,p_2)\in\mathbb{Z}^2$. 

Define
$\widetilde{g}: \mathbb{R}^2\rightarrow\mathbb{R}^2, (x,y)\mapsto
\widetilde{f^q}-(p_1,p_2)$.  Then we have 
$$\int_{\mathbb{T}^2}(\widetilde{g}(x,y)-(x,y))dxdy=0.$$ 

It is easy to see that $\widetilde{g}$ is a lift of diffeomorphism
$g:\mathbb{T}^2\rightarrow\mathbb{T}^2$ and $g$ is a Hamiltonian
diffeomorphism. Notice that $f^q=g$ on $\mathbb{T}^2$. According to
Asaoka-Irie's result (Theorem \ref{thm2.1}), for any nonempty open set
$U$, there exists a sequence $(g_{j})_{j\geqslant1}\rightarrow g$ such
that $\mathcal{P}(g_j)\cap U\neq\emptyset$ and
$supp(g^{-1}\circ g_{j})\subset U$.

This implies that given a symplectic diffeomorphism $f$ on
$\mathbb{T}^2$, with rotation vector $(\frac{p_1}{q}, \frac{p_2}{q})$,
there is a sequence of symplectic diffeomorphisms $(g_{j})_{j \geqslant
  1}$, such that
$(\widetilde{g}_{j})_{j\geqslant1}\rightarrow\widetilde{f}^q-(p_1,p_2)$,
and $(g_{j})_{j\geqslant1}\rightarrow f^{q}$ such that there are
periodic points in $U$. In other words, we can perturb
$f^q$ to create periodic points in any prescribed open set $U$.

However, the above argument does not prove our theorem. We need to
perturb $f$, instead of $f^q$, to creat periodic points. The problem
is that there may not be any map, say $f_j$, such that $f_j^q = g_j$.
 
To solve this problem, we apply the more general perturbation lemma,
Theorem \ref{thm2.3}, instead of the original theorem of Asaoka-Irie.

For any $U_0 \subset M$, suppose $f$ no periodic points in $U_0$.
Denote  $U_{1}=f^{-1}(U_{0})$, ..., $U_{q-1}=f^{-1}(U_{q-2})$. By
shrinking $U_0$, we may assume that $U_0, U_1, \ldots, U_{q-1}$ are
disjoint, $U_{i} \cap U_{j} = \emptyset$ for $i\neq j$. Let
$$U = U_0 \cup  U_1 \cup \ldots \cup U_{q-1}$$

For $t \in [0, 1]$, let $h_{t}: U_0 \rightarrow U_0$ be a family of
area-preserving diffeomorphism such that $h_{0} = \mbox{id}$ and total
action is nonzero at $t =1$, or,  $A(h_{1}) \neq 0$, as the example we
constructed in Section \ref{section3}. Consider the map
$$h_{t}\circ f: M \rightarrow M.$$
We have
\begin{itemize}
  \item For all $t \in [0, 1]$, $(h_{t} \circ f)^q \equiv f$ for all $f
    \in M \backslash (U_0 \cup U_1 \cup \ldots \cup U_{q-1}) = M
    \backslash U$.

    \item 
      $(h_{0}\circ f)^q = f^q $
      
    \item The total action of map $$(h_{1}\circ f)^q \circ f^{-q}: U
      \rightarrow U$$ is
$$A((h_{1}\circ f)^q \circ f^{-q}) = q
      A(h_{1}) \neq 0.$$
\end{itemize}

By Theorem \ref{thm2.3}, there is a $t \in [0, 1]$ such that
$$(h_{t}\circ f)^q: M \rightarrow M.$$
has a periodic point in $U$, which implies that
$$h_{t}\circ f: M \rightarrow M.$$ has a periodic point in $U_0$.

This proves the theorem.

% \bibliographystyle{plain}
%\bibliographystyle{acm}
%\bibliography{closing}


\end{document}